\documentclass[secthm,seceqn,amsthm,ussrhead,12pt]{amsart}
\usepackage{amsmath,latexsym}
\usepackage[english]{babel}
\usepackage[psamsfonts]{amssymb}
\usepackage{times}
\usepackage{cite}

\usepackage[mathcal]{euscript}
\numberwithin{equation}{section} \textwidth=17.5cm
\newtheorem{theorem}{Theorem}[section]
\newtheorem{lemma}[theorem]{Lemma}

\newtheorem{corollary}[theorem]{Corollary}

\newtheorem{remark}[theorem]{Remark}

\numberwithin{equation}{section}

\begin{document}

\title[Local derivations on subalgebras of $\tau$-measurable operators]{Local
derivations on subalgebras of $\tau$-measurable  operators with
respect to semi-finite von Neumann algebras}

\author{Farrukh Mukhamedov  and Karimbergen Kudaybergenov}

\address[Farrukh Mukhamedov]{Department of Computational \& Theoretical Sciences\\
Faculty of Science, International Islamic University Malaysia\\
P.O. Box, 141, 25710, Kuantan\\
Pahang, Malaysia} \email{{\tt far75m@yandex.ru} {\tt
farrukh\_m@iium.edu.my}}

\address[Karimbergen Kudaybergenov]{Department of Mathematics, Karakalpak state
university,  230113 Nukus, Uzbekistan.} \email{karim2006@mail.ru}

\maketitle
\begin{abstract}
This paper is devoted to local derivations on subalgebras on the
algebra $S(M, \tau)$ of all $\tau$-measurable  operators
affiliated with a  von Neumann algebra  $M$ without abelian
summands and with   a faithful normal semi-finite trace  $\tau.$
We prove that if $\mathcal{A}$ is  a solid $\ast$-subalgebra in
$S(M, \tau)$  such that $p\in \mathcal{A}$ for all projection
$p\in M$ with  finite trace, then every local derivation  on the
algebra $\mathcal{A}$ is a derivation. This result is new even in
the case standard subalgebras on  the algebra $B(H)$ of all
bounded linear operators on a Hilbert space $H.$ We also apply our
main theorem to the algebra $S_0(M, \tau)$ of all $\tau$-compact
operators affiliated with a semi-finite von Neumann algebra $M$
and with a faithful normal semi-finite trace $\tau.$
\end{abstract}

\maketitle

\bigskip

\section{Introduction}

\medskip

Given an algebra $\mathcal{A},$ a linear operator
$D:\mathcal{A}\rightarrow \mathcal{A}$ is called a
\textit{derivation}, if $D(xy)=D(x)y+xD(y)$ for all $x, y\in
\mathcal{A}$ (the Leibniz rule). Each element $a\in \mathcal{A}$
implements a derivation $D_a$ on $\mathcal{A}$ defined as
$D_a(x)=[a, x]=ax-xa,$ $x\in \mathcal{A}.$ Such derivations $D_a$
are said to be \textit{inner derivations}. If the element $a,$
implementing the derivation $D_a,$ belongs to a larger algebra
$\mathcal{B}$ containing $\mathcal{A},$ then $D_a$ is called
\textit{a spatial derivation} on $\mathcal{A}.$ A well known
direction in the study of the local action of derivations is the
local derivation problem. Recall that a linear map $\Delta$ of
$\mathcal{A}$  is called a \textit{local derivation} if for each
$x\in \mathcal{A},$  there exists a derivation
$D:\mathcal{A}\rightarrow \mathcal{A},$  depending on $x,$ such that
$\Delta(x)=D(x).$ This notion was introduced in 1990 independently
by Kadison \cite{Kad} and Larson and Sourour \cite{Lar}. In
\cite{Kad} it was proved that every norm continuous local derivation
from a von Neumann algebra into its dual normal bimodule is a
derivation. In \cite{Lar} the same result was obtained for the
algebra of all bounded linear operators acting on a Banach space.

In the last decade the structure of derivations and local
derivations on the algebra $LS(M)$ of all locally measurable
operators affiliated with a von Neumann algebra $M$  and on its
various subalgebras have been investigated by many authors (see
\cite{AAK2007, Alb2, Nur, AK2, AK1, AK2013, AK13, AAK13, AKNA,
Ber, BPS, Ber2, Ber3, Ber2013, Had}). In \cite{Nur} local
derivations have been investigated on the algebra $S(M)$ of all
measurable operators with respect a von Neumann algebra $M$.
 In particular, it was   proved
 that, for finite type I  von Neumann algebras without abelian
  direct summands, every local derivation
 on $S(M)$ is a derivation. Moreover, in the  case of abelian von Neumann algebra $M$
  necessary
and sufficient conditions have been  obtained for the algebra
$S(M)$ to admit local derivations which are not derivations. In
\cite{Had} local derivations have been investigated on the algebra
$S(M)$ for an arbitrary  von Neumann algebra $M$ and  it was
proved
 that for a  von Neumann algebras without abelian direct
  summands every local derivation
 on $S(M)$ is a derivation.
It should be noted that the proofs of the main result in the paper
\cite{Had}   are essentially based on the fact that the von
Neumann algebra $M$ is a subalgebra in the considered algebras.
Local and $2$-local maps have been studied on different operator
algebras by many authors \cite{Nur, JMAA, AAK13, AKNA, Bre1, Had,
Kad,  Lar}.

The present paper is devoted to local derivations on subalgebras
of algebra $S(M, \tau)$ of all $\tau$-measurable operators
affiliated with a
 von Neumann algebra $M$ with a faithful normal semi-finite trace $\tau.$
  Since  in general case we do not assumed that these subalgebras
  contain the von Neumann
algebra $M,$ one cannot directly apply the methods of the papers
\cite{Had} in this setting. Moreover, in our setting description of
local derivations is an open problem. Therefore the aim of this
paper to solve such a problem.

In Section 2 we give preliminaries from the theory of
$\tau$-measurable operators affiliated with a von Neumann algebra
$M.$

 In section 3 we
consider   a von Neumann algebra $M$ without abelian summands and
with   a faithful normal semi-finite trace  $\tau.$ We prove that
if $\mathcal{A}$ is  a solid $\ast$-subalgebra in $S(M, \tau)$
such that $p\in \mathcal{A}$ for all projection $p\in M$ with a
finite trace, then every local derivation $\Delta$ on the algebra
$\mathcal{A}$ is a derivation.

In section 4 we  apply the  main theorem of the previous section
to the Arens algebra and the algebra of all $\tau$-compact
operators  affiliated with a semi-finite von Neumann algebra $M$
and with a faithful normal semi-finite trace $\tau.$

\section{Algebras of $\tau$-measurable operators}

Let  $B(H)$ be the $\ast$-algebra of all bounded linear operators
on a Hilbert space $H,$ and let $\textbf{1}$ be the identity
operator on $H.$ Consider a von Neumann algebra $M\subset B(H)$
with the operator norm $\|\cdot\|$ and  with a faithful normal
semi-finite trace $\tau.$ Denote by $P(M)=\{p\in M:
p=p^2=p^\ast\}$ the lattice of all projections in $M$ and
$P_\tau(M)=\{p\in P(M): \tau(p)<+\infty\}.$

A linear subspace  $\mathcal{D}$ in  $H$ is said to be
\emph{affiliated} with  $M$ (denoted as  $\mathcal{D}\eta M$), if
$u(\mathcal{D})\subset \mathcal{D}$ for every unitary  $u$ from
the commutant
$$M'=\{y\in B(H):xy=yx, \,\forall x\in M\}$$ of the von Neumann algebra $M.$

A linear operator  $x: \mathcal{D}(x)\rightarrow H,$ where the
domain  $\mathcal{D}(x)$ of $x$ is a linear subspace of $H,$ is
said to be \textit{affiliated} with  $M$ (denoted as  $x\eta M$)
if $\mathcal{D}(x)\eta M$ and $u(x(\xi))=x(u(\xi))$
 for all  $\xi\in
\mathcal{D}(x)$  and for every unitary  $u\in M'.$

A linear subspace $\mathcal{D}$ in $H$ is said to be
\textit{strongly dense} in  $H$ with respect to the von Neumann
algebra  $M,$ if
\begin{enumerate}
\item[1.] $\mathcal{D}\eta M;$

\item[2.] there exists a sequence of projections $\{p_n\}_{n=1}^{\infty}$
in $P(M)$  such that $p_n\uparrow\textbf{1},$ $p_n(H)\subset
\mathcal{D}$ and $p^{\perp}_n=\textbf{1}-p_n$ is finite in  $M$
for all $n\in\mathbb{N}$.
\end{enumerate}
A closed linear operator  $x$ acting in the Hilbert space $H$ is
said to be \textit{measurable} with respect to the von Neumann
algebra  $M,$ if
 $x\eta M$ and $\mathcal{D}(x)$ is strongly dense in  $H.$

 Denote by $S(M)$  the set of all linear operators on $H,$ measurable with
respect to the von Neumann algebra $M.$ If $x\in S(M),$
$\lambda\in\mathbb{C},$ where $\mathbb{C}$  is the field of
complex numbers, then $\lambda x\in S(M)$  and the operator
$x^\ast,$  adjoint to $x,$  is also measurable with respect to $M$
(see \cite{Seg}). Moreover, if $x, y \in S(M),$ then the operators
$x+y$  and $xy$  are defined on dense subspaces and admit closures
that are called, correspondingly, the strong sum and the strong
product of the operators $x$  and $y,$  and are denoted by
$x\stackrel{.}+y$ and $x \ast y.$ It was shown in \cite{Seg} that
$x\stackrel{.}+y$ and $x \ast y$ belong to $S(M)$ and these
algebraic operations make $S(M)$ a $\ast$-algebra with the
identity $\textbf{1}$  over the field $\mathbb{C}.$ Here, $M$ is a
$\ast$-subalgebra of $S(M).$ In what follows, the strong sum and
the strong product of operators $x$ and $y$  will be denoted in
the same way as the usual operations, by $x+y$  and $x y.$

It is clear that if the von Neumann algebra $M$ is finite then every linear operator
affiliated with $M$ is measurable and, in particular, a self-adjoint operator is
measurable with respect to $M$ if and only if all its
 spectral projections belong to $M$.

 Let   $\tau$ be a faithful normal semi-finite trace on
 $M.$ We recall that a closed linear operator
  $x$ is said to be  $\tau$\textit{-measurable} with respect to the von Neumann algebra
   $M,$ if  $x\eta M$ and   $\mathcal{D}(x)$ is
  $\tau$-dense in  $H,$ i.e. $\mathcal{D}(x)\eta M$ and given   $\varepsilon>0$
  there exists a projection   $p\in M$ such that   $p(H)\subset\mathcal{D}(x)$
  and $\tau(p^{\perp})<\varepsilon.$
   Denote by  $S(M,\tau)$ the set of all   $\tau$-measurable operators affiliated with $M.$

Note that if the trace $\tau$ is finite then $S(M,\tau)=S(M).$

    Consider the topology  $t_{\tau}$ of convergence in measure or \textit{measure topology}
    on $S(M, \tau),$ which is defined by
 the following neighborhoods of zero:
$$V(\varepsilon, \delta)=\{x\in S(M, \tau): \exists\, e\in P(M),
 \tau(e^{\perp})<\delta, xe\in
M,  \|xe\|<\varepsilon\},$$  where $\varepsilon, \delta$ are
positive numbers.

 It is well-known \cite{Nel} that $S(M, \tau)$ equipped with the measure topology is a complete
metrizable topological $\ast$-algebra.

\section{Local derivations on $\ast$-subalgebras  of $\tau$-measurable operators}

\medskip

Recall that $\ast$-subalgebra $\mathcal{A}\subset S(M,\tau)$ is
said to be solid, if $x\in S(M, \tau)$ and $y\in \mathcal{A}$
satisfy $|x|\leq |y|,$ then $x\in \mathcal{A}.$

The main result of this section is the following

\begin{theorem}\label{main}
 Let $M$ be a
semi-finite von Neumann algebra without abelian direct summands
and let $\tau$ be a faithful normal semi-finite trace on $M.$
Suppose that $\mathcal{A}$ is  a solid $\ast$-subalgebra in $S(M,
\tau)$ such that  $p\in \mathcal{A}$ for all $p\in P_\tau(M).$
Then every local derivation $\Delta$ on the algebra $\mathcal{A}$
is a derivation.
\end{theorem}

For the proof of this theorem we need several lemmata.

Let  $M$ be a von Neumann algebra  and let  $x\in S(M, \tau).$ The
projection
$$
r(x)=\inf\{p\in P(M): xp=x\}
$$
is called the   \emph{right support} of the element  $x,$ and the
projection
$$
 l(x)=\inf\{p\in P(M): px=x\}
 $$
is called its  \emph{left support}. The projection $s(x)=r(x)\vee
l(x)$ is called the support of the element $x.$ For
$\ast$-subalgebra $\mathcal{A}\subset S(M, \tau)$ denote
$$
\mathcal{F}_{\tau}(\mathcal{A})=\{x\in \mathcal{A}: s(x)\in
P_{\tau}(M)\}.$$

The following properties of $\mathcal{F}_{\tau}(\mathcal{A})$
directly follow  from the definition.

\begin{lemma} Let  $\mathcal{A}$ be a $\ast$-subalgebra in $S(M, \tau).$ Then the following
assertions are equivalent:
\begin{enumerate}
\item  $x\in \mathcal{F}_{\tau}(\mathcal{A});$
\item  $\exists\, p\in P_{\tau}(M)$  such that  $px=x;$
\item $\exists\, p\in P_{\tau}(M)$ such that $xp=x;$
\item $\exists\, p\in P_{\tau}(M)$ such that $pxp=x.$
\end{enumerate}
\end{lemma}

From this lemma we immediately get

\begin{corollary}
  $\mathcal{F}_{\tau}(\mathcal{A})$ is an ideal in
 $\mathcal{A}.$
\end{corollary}

\begin{lemma} Let $\Delta$ be a local derivation on $\mathcal{A}$. Then $\Delta(\mathcal{F}_{\tau}(\mathcal{A}))\subset
\mathcal{F}_{\tau}(\mathcal{A})$.
\end{lemma}

\begin{proof} Take any $x\in\mathcal{F}_{\tau}(\mathcal{A}).$ Then due to locality of $\Delta$ one can
find a derivation $D$ on $\mathcal{A}$ such that $\Delta(x)=D(x)$.
It is clear that
$$
l(D(x)s(x))\preceq s(x),
$$
$$
r(xD(s(x)))\preceq s(x),
$$
where $p\preceq q$ means that $p$ is equivalent  to a
subprojection of the projection $q.$  Since
$$
D(x)=D(xs(x))=D(x)s(x)+xD(s(x))
$$
we have
\begin{eqnarray*}
\tau(s(D(x))) & =& \tau(s(D(x)s(x)+xD(s(x))))\leq \\
& \leq & \tau(s(x)\vee l(D(x)s(x))\vee
r(xD(s(x))))\leq\\
& \leq & \tau(s(x))+\tau(s(x))+\tau(s(x))=3\tau(s(x)),
\end{eqnarray*}
i.e.
$$
\tau(s(D(x))) \leq 3\tau(s(x)).
$$
Thus $D(x)\in \mathcal{F}_{\tau}(\mathcal{A})$, so $\Delta(x)\in
\mathcal{F}_{\tau}(\mathcal{A})$.

Therefore, $\Delta$ maps $\mathcal{F}_{\tau}(\mathcal{A})$ into
itself.
\end{proof}

Let $p\in \mathcal{A}$ be a projection. One can see that the mapping
\begin{equation}\label{reduc}
 D^{(p)}:x\rightarrow pD(x)p,\, x\in p\mathcal{A}p
\end{equation}
is a derivation on $p\mathcal{A}p.$ Now let $\Delta$ be a local
derivation on $\mathcal{A}.$ Take an element $x\in \mathcal{A}$
and a derivation $D$ on $\mathcal{A}$ such that
$\Delta(pxp)=D(pxp).$ Then
$$
p\Delta(pxp)p=pD(pxp)p=D^{(p)}(pxp).
$$
This means that the mapping $\Delta^{(p)}$ defined similar to
\eqref{reduc} is a local derivation on $p\mathcal{A}p.$

\begin{lemma}\label{rest}
If  $\Delta$ is  a local derivation on $\mathcal{A},$ then the
restriction $\Delta|_{\mathcal{F}_{\tau}(\mathcal{A})}$ is a
derivation.
\end{lemma}

\begin{proof}
Let $x, y \in \mathcal{F}_{\tau}(\mathcal{A}).$ Denote
$$
p=s(x)\vee s(y)\vee s(xy)\vee s(\Delta(x))\vee s(\Delta(y))\vee
s(\Delta(xy)).
$$
Since $\mathcal{F}_{\tau}(\mathcal{A})$ is an ideal in
$\mathcal{A}$ and $\Delta$ maps $\mathcal{F}_{\tau}(\mathcal{A})$
into itself, we obtain that the projection $p\in P_{\tau}(M).$
Consider the local derivation $\Delta^{(p)}$ on $p\mathcal{A}p.$
Since $p\in \mathcal{A}$ and $\mathcal{A}$ is a solid
$\ast$-subalgebra in $S(M, \tau)$ we get $pMp\subseteq
\mathcal{A}.$ Furthermore, $\mathcal{A}$ has no abelian direct
summands, and therefore \cite[Theorem 1]{Had} implies that
$\Delta^{(p)}$ is a derivation. Taking into account that $x, y\in
p\mathcal{A}p$ we obtain
$$
\Delta^{(p)}(xy)=\Delta^{(p)}(x)y+x\Delta^{(p)}(y).
$$
By construction of the projection $p$ we have
\begin{align*}
\Delta(xy)  = \Delta^{(p)}(xy) & =\Delta^{(p)}(x)y+x\Delta^{(p)}(y)= \\
& =  \Delta(x)y+x\Delta(y).
\end{align*}
This means that $\Delta$ is a derivation on
$\mathcal{F}_{\tau}(\mathcal{A}).$ This completes the proof.
\end{proof}

\begin{remark}\label{anni} Let $y\in \mathcal{A}$ and $yp=0$ for all $p\in
P_\tau(M).$   Since the map $x\mapsto  yxy^\ast$ is positive and
monotone continuous, taking $p\uparrow \mathbf{1}$ in
$ypy^\ast=0,$ we obtain that $yy^\ast = 0.$ Therefore $y = 0.$
\end{remark}

\begin{proof}[Proof of Theorem~\ref{main}]
 We shall show that
$$
\Delta (xy)=\Delta(x)y+x\Delta(y)
$$
for all $x, y\in \mathcal{A}.$ We consider the following two cases.

{\sc Case 1}. Let $x\in \mathcal{F}_{\tau}(\mathcal{A})$ and $y\in
\mathcal{A}.$ Since $\mathcal{F}_{\tau}(\mathcal{A})$ is an ideal in
$\mathcal{A}$ and $\Delta$ maps $\mathcal{F}_{\tau}(\mathcal{A})$
into itself, we obtain that the projection
$$
p=s(xy)\vee s(\Delta(xy))\vee s(\Delta(x)y)\vee s(x\Delta(y))
$$ has a finite trace. Taking into
account the equalities
$$
xyp=xy, \, \Delta(xy)p=\Delta(xy)
$$
and Lemma~\ref{rest} we obtain
\begin{align*}
\Delta(xy)   & =\Delta(xyp)= \Delta(xy)p+xy\Delta(p)=\\
 & = \Delta(xy)+xy\Delta(p),
\end{align*}
i.e. $xy\Delta(p)=0.$

Further
\begin{align*}
x\Delta(yp)p   & =x\Delta(y pp)-xyp\Delta(p)=\\
&= x\Delta(yp)-xy\Delta(p)= x\Delta(yp),
\end{align*}
i.e.
\begin{equation}\label{pp}
    x\Delta(yp)p = x\Delta(yp).
\end{equation}
Now taking into account \eqref{pp}, the equalities
$$
xy(\mathbf{1}-p)=0, \, x\Delta(y)p=x\Delta(y)
$$
and  the linearity of $\Delta$ we have
\begin{align*}
x\Delta(y)   & =x\Delta(y)p=x\Delta(yp+y(\mathbf{1}-p))p=\\
& = x\Delta(yp)p+x\Delta(y(\mathbf{1}-p))p=
x\Delta(yp)p+x D(y(\mathbf{1}-p))p= \\
 & = x\Delta(yp)p+x D(y(\mathbf{1}-p)p)-xy (\mathbf{1}-p)D(p)=
  x\Delta(yp),
\end{align*}
where $D$ is a derivation on $\mathcal{A}$ such that
$\Delta(y(\mathbf{1}-p))=D(y(\mathbf{1}-p)).$ Consequently
\begin{align*}
x\Delta(yp)   = x\Delta(y).
\end{align*}

Finally we obtain that
\begin{align*}
\Delta(xy)   & =\Delta(xyp)=   \Delta(x)yp+x\Delta(yp)= \\
& =   \Delta(x) y+x\Delta(y).
\end{align*}

Similar  as above we can check the case $x\in \mathcal{A}$ and
$y\in \mathcal{F}_{\tau}(\mathcal{A}).$

{\sc Case 2}. Let $x, y\in \mathcal{A}$ be arbitrary elements. Take
an arbitrary $q\in P_{\tau}(M).$ By the case 1 we have
\begin{align*}
\Delta(y)q   & =\Delta(yq)-y\Delta(q).
\end{align*}
Taking into account this equality and the case 1 we obtain
\begin{align*}
\Delta(xy)q   & =\Delta(xyq)-  xy \Delta(q)= \\
& =   \Delta(x)yq+x\Delta(yq)-xy\Delta(q)=\\
& =   \Delta(x)yq+x[\Delta(yq)-y\Delta(q)]=\\
& =   \Delta(x)yq+x\Delta(y)q,
\end{align*}
i.e.
\begin{align*}
\Delta(xy)q   & =  [\Delta(x)y+x\Delta(y)]q
\end{align*}
for all $q\in P_\tau(M).$ Taking into account Remark~\ref{anni} we
obtain
\begin{align*}
\Delta(xy)   & =  \Delta(x)y+x\Delta(y).
\end{align*}
The proof is complete.
\end{proof}

We stress that Theorem~\ref{main} is new even in the case of type
I$_\infty$ von Neumann factors.

For a Hilbert space $H$ by $\mathcal{F}(H)$  we denote  the
algebra of all finite rank operators in $B(H).$ Recall that a
standard operator algebra is any subalgebra of $B(H)$ which
contains $\mathcal{F}(H).$

Theorem~\ref{main} implies the following result.

\begin{corollary}
Let $H$ be a Hilbert space and let $\mathcal{U}$ be a standard
algebra in $B(H).$  Then any local derivation
$\Delta:\mathcal{U}\rightarrow \mathcal{U}$ is a spatial
derivation and implemented by an element $a\in B(H).$
\end{corollary}

\begin{remark} A similar result for  local derivations on $B(X),$
where $X$ is a Banach space, has been obtained in \cite[Corollar
3.7]{Bre1} under the additional assumption
 of continuity of the map with respect to the weak operator topology.
\end{remark}

\begin{remark} We note that if one replaces $S(M, \tau)$ with $S(M)$ all the results will remain
true. In this case $\mathcal{F}_{\tau}(\mathcal{A})$ is replaced by
the set of finite projections of $\mathcal{A}$ and instead of $\tau$
is used the dimension function.
\end{remark}

\section{Local derivations on algebra $\tau$-compact operators and Arens algebras}

\medskip

In this section we shall consider a local derivations on algebras
$\tau$-compact operators and on Arens algebras, respectively.

\subsection{algebra of $\tau$-compact operators} In this subsection we shall
consider an algebra of
$\tau$-compact operators.

In the algebra $S(M, \tau)$ consider the subset $S_0(M, \tau)$ of
all operators $x$ such that given any $\varepsilon>0$ there is a
projection $p\in P(M)$ with $\tau(p^{\perp})<\infty,\,xp\in M$ and
$\|xp\|<\varepsilon.$ The elements of $S_0(M, \tau)$ is called
\textit{$\tau$-compact operators} with respect to $M$ and $\tau.$
It is known \cite{Mur} that  $S_0(M, \tau)$ is a solid
$\ast$-subalgebra in $S(M, \tau)$ and a bimodule over $M,$ i.e.
$ax, xa\in S_0(M, \tau)$ for all $x\in S_0(M, \tau)$ and $a\in M.$
Note that if $M=B(H)$ and $\tau=tr,$ where $tr$ is the canonical
trace on $B(H),$ then $S_0(M, \tau)=K(H),$ where $K(H)$ is the
ideal of compact operators from $B(H).$

 The following properties of the algebra $S_0(M, \tau)$
are known (see  \cite{Str}):

Let  $M$ be a von Neumann algebra with a faithful normal
semi-finite trace $\tau.$ Then

1) $S(M, \tau)=M+S_0(M, \tau);$

2) $S_0(M, \tau)$ is an ideal in $S(M, \tau).$

Note that if the trace $\tau$ is finite then
$$S_0(M, \tau)=S(M, \tau)=S(M).$$
 It is well-known \cite{Str} $S_0(M, \tau)$ equipped with the measure topology is a complete
metrizable topological $\ast$-algebra.

It is clear that  $p\in S_0(M, \tau)$ for all $p\in P_\tau(M).$

Theorem~\ref{main} implies the following result.

\begin{theorem}\label{macom}
 Let $M$ be a
semi-finite von Neumann algebra without abelian direct summands
and let $\tau$ be a faithful normal semi-finite trace on $M.$ Then
every local derivation $\Delta$ on the algebra $S_0(M, \tau)$ is a
derivation.
\end{theorem}

\begin{remark}
If  $M$ is  an abelian von Neumann algebra with a faithful normal
semi-finite trace $\tau$ such that  the lattice $P(M)$ of
projections in $M$ is not atomic, then the algebra $S_0(M, \tau)$
admits a local derivation which is not a derivation (see
\cite[Theorem 3.2]{Nur}).
\end{remark}

In  \cite[Theorem 4.9]{AK13} it was proved that in the case when
$M$ is a properly infinite von Neumann algebra with a faithful
normal semi-finite trace $\tau,$ then any derivation $D$ on
$S_0(M, \tau)$  is a  spatial derivation and implemented by an
element $a\in S(M, \tau).$ Therefore Theorem~\ref{macom} implies
that

\begin{theorem}\label{proper}
 If $M$ is a properly infinite von Neumann algebra with a faithful normal
semi-finite trace $\tau,$ then any local derivation $\Delta:S_0(M,
\tau)\rightarrow S_0(M, \tau)$  is a  spatial derivation and
implemented by an element $a\in S(M, \tau).$
\end{theorem}

\subsection{Arens algebras} Now we are going to consider Arens algebras
associated with a von Neumann
algebra and  a semi-finite faithful normal trace.

Let $M$ be a von Neumann algebra with a faithful normal
semi-finite trace $\tau.$

Take $x\in S(M, \tau), x\geq 0$  and let $x = \int\limits_0^\infty
\lambda\, d e_\lambda$  be it spectral resolution. Denote
$\tau(x)=\sup\limits_{n\geq1} \int\limits_0^n \lambda\, d
\tau(e_\lambda).$

Given $p\geq1$ put $L^{p}(M, \tau)=\{x\in S(M,
\tau):\tau(|x|^{p})<\infty\}.$ It is known \cite{Mur} that
$L^{p}(M, \tau)$ is a Banach space with respect to the norm
$$\|x\|_p=(\tau(|x|^{p}))^{1/p},\quad x\in L^{p}(M, \tau).$$
  Consider the intersection
\begin{center}
$L^{\omega}(M, \tau)=\bigcap\limits_{p\geq1}L^{p}(M, \tau).$
\end{center}
It is proved in \cite{Abd} that  $L^{\omega}(M, \tau)$ is a
locally convex complete metrizable   $\ast$-algebra with respect
to the  topology $t$ generated by the family of norms
$\{\|\cdot\|_p\}_{p\geq1}.$
 The algebra
$L^{\omega}(M, \tau)$ is called a (non commutative) \textit{Arens
algebra}.

Note that $L^{\omega}(M, \tau)$ is  a solid $\ast$-subalgebra in
$S(M, \tau)$ and if $\tau$ is a finite trace then $M\subset
L^\omega(M, \tau).$

Further  consider the following spaces
$$L^{\omega}_2(M, \tau)=\bigcap\limits_{p\geq 2}L^p(M, \tau)$$
and
$$
M+ L^{\omega}_2(M, \tau)=\{x+y: x\in M, y\in L^{\omega}_2(M,
\tau)\}.
$$

It is known \cite{AAK2007} that $L^{\omega}_2(M, \tau)$  and $M+
L^{\omega}_2(M, \tau)$ are a $\ast$-algebras and $L^{\omega}(M,
\tau)$ is an ideal in $M+L^{\omega}_2(M, \tau).$

Note that if $\tau(\textbf{1})<\infty$ then $M+ L^{\omega}_2(M,
\tau)=L^{\omega}_2(M, \tau)=L^{\omega}(M, \tau).$

It is known \cite[Theorem 3.7]{AAK2007} that if
   $M$ is a  von Neumann algebra with a faithful normal
semi-finite trace $\tau$ then any derivation $D$ on
  $L^{\omega}(M, \tau)$ is spatial, moreover it is
  implemented by an element of $M+L^{\omega}_2(M, \tau),$ i. e.
\begin{equation}
\label{SP} D(x)=ax-xa, \quad x\in L^{\omega}(M, \tau)
\end{equation}
for some $a\in M+L^{\omega}_2(M, \tau).$ In particular, if $M$ is
abelian, then any derivation on $L^\omega(M, \tau)$ is zero.

Note that  $p\in L^\omega(M, \tau)$ for all $p\in P_\tau(M).$

We need the following auxiliary result.

\begin{lemma} \label{C}
Let   $M$ be a  semi-finite von Neumann algebra  with a faithful
normal semi-finite trace $\tau$ and with the center $Z(M).$ Then
every local derivation $\Delta$ on the algebra $L^\omega(M, \tau)$
is necessarily $P(Z(M))$-homogeneous, i.e.
$$
\Delta(zx)=z\Delta(x)
$$ for any central projection $z\in
P(Z(M))=P(M)\cap Z(M)$ and for all $x\in L^\omega(M, \tau).$
\end{lemma}

\begin{proof}
Take $z\in P(Z(M))$ and $x\in L^\omega(M, \tau).$ For the element
$zx$ by the definition of the local derivation $\Delta$ there
exists a derivation $D_a$ on $L^\omega(M, \tau)$ of the
form~\eqref{SP}  such that $\Delta(zx)=D_a(zx).$ Since the
projection $z$ is central, one has that
$$
D_a(zx)=[a, zx]=z[a,x]=zD_a(x).
$$
  Multiplying the equality $\Delta(zx)=D_a(zx)$  by $z$ we obtain
  $$
  z\Delta(zx)=zD_a(zx)=zD_a(x)=D_a(zx)=\Delta(zx),$$  i.e.
$$
(\mathbf{1}-z)\Delta(zx)=0.
$$
Replacing $z$ by $\mathbf{1}-z$ one finds
$$
z\Delta((\mathbf{1}-z)x)=0.$$ Therefore by the linearity of
$\Delta$ we have
$$
z\Delta(x)=z\Delta(zx)+z\Delta((\mathbf{1}-z)x)=z\Delta(zx)=\Delta(zx),
$$
and thus $z\Delta(x)=\Delta(zx).$ The proof is complete.
\end{proof}

\begin{theorem}\label{T1}
Let   $M$ be a  semi-finite von Neumann algebra  with a faithful
normal semi-finite trace $\tau.$ Then any local derivation
$\Delta$ on the algebra  $L^{\omega}(M,\tau)$ is a spatial
derivation of the form (\ref{SP}).
\end{theorem}

\begin{proof} Let   $M$ be a  semi-finite von Neumann algebra.
There exist mutually orthogonal central projections $z_1, z_2$ in
$M$ with $z_1+z_2=\mathbf{1}$ such that
\begin{itemize}
\item $z_1 M$ is abelian;
\item $z_2 M$ has no abelian summands.
\end{itemize}
By Lemma~\ref{C} the operator $\Delta$ maps $z_iL^\omega(M,
\tau)\equiv L^\omega(z_iM, \tau_i)$ into itself for $i=1,2,$ where
$\tau_i$ is the restriction of $\tau$ on $z_iM (i=1,2).$  As it
was mentioned above $z_1 \Delta$ is zero. By Theorem~\ref{main} we
obtain that $\Delta=z_2\Delta$ is a derivation. The proof is
complete.
\end{proof}

Note that if the trace $\tau$ is  finite Theorem~\ref{T1}  it was
given in \cite[Theorem 2.1]{AKNA}.

\section*{Acknowledgments}

The second named author (K.K.)
 acknowledges the MOHE grant ERGS13-024-0057 for support, and
International Islamic University Malaysia for kind hospitality.

\end{document}